\documentclass[a4paper,11pt]{article}
\usepackage{graphicx,epsfig}
\usepackage{fancyhdr,fancybox}
\usepackage{indentfirst}
\usepackage{titlesec}
\usepackage{verbatim}
\usepackage[sort&compress, numbers]{natbib}
\usepackage{array,dcolumn,tabularx}
\usepackage{setspace}
\usepackage{geometry}
\usepackage{extarrows,chemarrow,xypic} 
\usepackage[small]{caption2}
\usepackage{sectsty}
\usepackage{microtype}
\DisableLigatures[f]{encoding = *, family = * }

\usepackage{times}     

\usepackage{latexsym}
\usepackage{amsmath}   
\usepackage{amssymb}   
\usepackage{amsbsy}
\usepackage{amsthm}
\usepackage{amsfonts}
\usepackage{mathrsfs}  
\usepackage{bm}        
\usepackage{relsize}   
\usepackage{hyperref}  

\newcommand{\paperfont}{\fontsize{12pt}{1.3\baselineskip}\selectfont}
\sectionfont{\fontsize{13}{15}\selectfont}
\begin{document}


\theoremstyle{definition}
\makeatletter
\thm@headfont{\bf}
\makeatother
\newtheorem{theorem}{Theorem}[section]
\newtheorem{definition}[theorem]{Definition}
\newtheorem{lemma}[theorem]{Lemma}
\newtheorem{proposition}[theorem]{Proposition}
\newtheorem{corollary}[theorem]{Corollary}
\newtheorem{remark}[theorem]{Remark}
\newtheorem{example}[theorem]{Example}
\newtheorem{assumption}[theorem]{Assumption}
\numberwithin{equation}{section}

\lhead{}
\rhead{}
\lfoot{}
\rfoot{}

\renewcommand{\refname}{References}
\renewcommand{\figurename}{Figure}
\renewcommand{\tablename}{Table}
\renewcommand{\proofname}{Proof}

\newcommand{\dnumiag}{\mathrm{diag}}
\newcommand{\tr}{\mathrm{tr}}
\newcommand{\dnum}{\mathrm{d}}

\newcommand{\Enum}{\mathbb{E}}
\newcommand{\Pnum}{\mathbb{P}}
\newcommand{\Rnum}{\mathbb{R}}
\newcommand{\Cnum}{\mathbb{C}}
\newcommand{\Znum}{\mathbb{Z}}
\newcommand{\Nnum}{\mathbb{N}}
\newcommand{\abs}[1]{\left\vert#1\right\vert}
\newcommand{\set}[1]{\left\{#1\right\}}
\newcommand{\norm}[1]{\left\Vert#1\right\Vert}
\newcommand{\innp}[1]{\langle {#1}]}

\title{\textbf{Moderate maximal inequalities for the Ornstein-Uhlenbeck process}}
\author{Chen Jia$^{1}$,\;\;\;Guohuan Zhao$^{2,*}$ \\
\footnotesize $^1$Department of Mathematical Sciences, University of Texas at Dallas, Richardson, TX 75080, U.S.A.\\
\footnotesize $^2$Academy of Mathematics and Systems Science, Chinese Academy of Sciences, Beijing 100190, China.\\
\footnotesize $^*$Author to whom correspondence should be addressed. \\
\footnotesize E-mail: jiac@utdallas.edu (C. Jia), zhaoguohuan@gmail.com (G. Zhao)\\}
\date{}                              
\maketitle                           
\thispagestyle{empty}                

\paperfont

\begin{abstract}
The maximal inequalities for diffusion processes have drawn increasing attention in recent years. However, the existing proof of the $L^p$ maximum inequalities for the Ornstein-Uhlenbeck process was dubious. Here we give a rigorous proof of the moderate maximum inequalities for the Ornstein-Uhlenbeck process, which include the $L^p$ maximum inequalities as special cases and generalize the remarkable $L^1$ maximum inequalities obtained by Graversen and Peskir [P. Am. Math. Soc., 128(10):3035-3041, 2000]. As a corollary, we also obtain a new moderate maximal inequality for continuous local martingales, which can be viewed as a supplement of the classical Burkholder-Davis-Gundy inequality. \\

\noindent 
\textbf{Keywords}: moderate function, law of the iterated logarithm, good $\lambda$ inequality \\

\noindent
\textbf{AMS Subject Classifications}: 60H10, 60J60, 60J65, 60G44
\end{abstract}

\section{Introduction}
The Ornstein-Uhlenbeck process, which is the solution to the Langevin equation
\begin{equation*}
dX_t = -\alpha X_tdt+dW_t,\;\;\;X_0 = 0,
\end{equation*}
is one of the most important kinetic models in statistical mechanics, where $\alpha>0$ and $W$ is a standard Wiener process. It describes the velocity of an \emph{underdamped} Brownian particle or the position of an \emph{overdamped} Brownian particle driven by the harmonic potential. An important question is how far the Brownian particle can travel before a given time. This problem is closely related to the maximal inequalities in probability theory.

The $L^p$ maximal inequalities for martingales are one of the classical results in probability theory. Let $M$ be a continuous local martingale vanishing at zero. The Burkholder-Davis-Gundy (BDG) inequality \cite{revuz1999continuous} claims that for any $p>0$, there exist two positive constants $c_p$ and $C_p$ such that for any stopping time $\tau$ of $M$,
\begin{equation*}
c_p\Enum[M]^{p/2}_\tau \leq \Enum[\sup_{0\leq t\leq\tau}|M_t|]^p \leq C_p\Enum[M]^{p/2}_\tau.
\end{equation*}
where $[M]$ is the quadratic variation process of $M$ (see \cite[Chapter IV, Exercise 4.25]{revuz1999continuous} for a moderate version of this inequality). 

Over the past two decades, significant progress has been made in the maximal inequalities for diffusion processes \cite{graversen1998maximal, graversen1998optimal, graversen2000maximal, peskir2001bounding, yan2004ratio, yan2005lp, yan2005lpestimates, lyulko2014sharp}. In particular, Graversen and Peskir \cite{graversen2000maximal} proved the following 
remarkable 
$L^1$ maximum inequality for the Ornstein-Uhlenbeck process by using Lenglart's domination principle: there exist two positive constants $c$ and $C$ independent on $\alpha$ such that for any stopping time $\tau$ of $X$,
\begin{equation*}
\frac{c}{\alpha^{1/2}}\Enum\log^{1/2}(1+\alpha\tau) \leq \Enum[\sup_{0\leq t\leq\tau}|X_t|] \leq \frac{C}{\alpha^{1/2}}\Enum\log^{1/2}(1+\alpha\tau).
\end{equation*}
Subsequently, Peskir \cite{peskir2001bounding} established the $L^1$ maximum inequalities for a large class of diffusion processes and obtained satisfactory results. However, the $L^p$ maximum inequalities for diffusion processes turn out to be more difficult, even for the Ornstein-Uhlenbeck process. As an attempt, Yan et al. \cite{yan2004ratio, yan2005lp, yan2005lpestimates} studied the $L^p$ maximum inequalities for a class of diffusion processes. Although their ideas are fairly nice, their detailed proofs are questionable because they mistakenly regarded the random time $TI_{\{S<T\}}$ as a stopping time, where $S$ and $T$ are two stopping times with $S\leq T$ \cite[Page 6, Lines 2 and 10]{yan2005lp}.

In this paper, we give a rigorous proof of the \emph{moderate maximum inequalities} for the Ornstein-Uhlenbeck process, which include the $L^p$ maximum inequalities as special cases. 
Our method is based on using ``good-$\lambda$" inequalities  (Lemma \ref{goodlambda}) introduced by Burkerholder (cf. \cite{burholderdistribution}). 
As a corollary, we also obtain a new moderate maximal inequality for the Brownian motion and general continuous local martingales, which can be viewed as a good supplement to the classical BDG inequality.

\section{Results}
Let $X$ be the one-dimensional Ornstein-Uhlenbeck process starting from zero, which is the solution to the stochastic differential equation
\begin{equation}\label{system}
dX_t = -\alpha X_tdt+dW_t,\;\;\;X_0 = 0,
\end{equation}
where $\alpha>0$ and $W$ is a standard Brownian motion defined on some filtered probability space $(\Omega,\{\mathscr{F}_t\},\Pnum)$. 
We next introduce the conception of moderate function(see also  \cite[p.164, line 11]{revuz1999continuous} ). 

\begin{definition}
A function $f:\Rnum_+\rightarrow\Rnum_+$ is called \emph{moderate} if
\begin{itemize}\setlength{\itemsep}{1pt}\setlength{\parsep}{1pt}\setlength{\parskip}{1pt}
\item[(a)] it is a continuous increasing function vanishing at zero and
\item[(b)] there exists $\lambda>1$ and $\gamma<\infty$ such that
\begin{equation}\label{requirement}
f(\lambda t)\leq \gamma f(t)\quad \mbox{for all} \ t>0. 
\end{equation}
\end{itemize}
\end{definition}

In particular, $f(t) = t^p$ is a moderate function for any $p>0$. Our main result is the following moderate maximum inequality.
\begin{theorem}\label{main}
For any moderate function $f$, there exists two positive constants $c_{\alpha,f}$ and $C_{\alpha,f}$ such that for any stopping time $\tau$ with respect to $\{\mathscr{F}_t\}$,
\begin{equation*}
c_{\alpha,f}\Enum f(\log^{1/2}(1+\alpha\tau))
\leq \Enum\left[\sup_{0\leq t\leq\tau}f(|X_t|)\right]
\leq C_{\alpha,f}\Enum f(\log^{1/2}(1+\alpha\tau)).
\end{equation*}
In particular, for any $p>0$, there exists two positive constants $c_p$ and $C_p$ independent of $\alpha$ such that for any stopping time $\tau$ of $X$,
\begin{equation*}
\frac{c_p}{\alpha^{p/2}}\Enum\log^{p/2}(1+\alpha\tau)
\leq \Enum\left[\sup_{0\leq t\leq\tau}|X_t|^p\right]
\leq \frac{C_p}{\alpha^{p/2}}\Enum\log^{p/2}(1+\alpha\tau).
\end{equation*}
\end{theorem}

The above theorem implies two moderate maximum inequalities for the Brownian motion and continuous local martingale which are in line with
Corollary 2.7 and 2.8 in \cite{graversen2000maximal}. For the reader's convenience, we provide short proofs here.
\begin{corollary}\label{BM}
For any moderate function $f$, there exists two positive constants $c_f$ and $C_f$ such that for any stopping time $\tau$ of $W$,
\begin{equation*}
\begin{split}
c_f\Enum f\big(\log^{1/2}(1+\log(1+\tau))\big)
&\leq \Enum\left[\sup_{0\leq t\leq\tau}f\left(\frac{|W_t|}{\sqrt{1+t}}\right)\right]\\
&\leq C_f\Enum f\big(\log^{1/2}(1+\log(1+\tau))\big).
\end{split}
\end{equation*}
In particular, for any $p>0$, there exists two positive constants $c_p$ and $C_p$ such that for any stopping time $\tau$ of $W$,
\begin{equation*}
c_p\Enum\log^{p/2}(1+\log(1+\tau)) \leq \Enum\left[\sup_{0\leq t\leq\tau}\frac{|W_t|^p}{(1+t)^{p/2}}\right] \leq C_p\Enum\log^{p/2}(1+\log(1+\tau)).
\end{equation*}
\end{corollary}

\begin{proof}
It is easy to check that \eqref{system} has the following explicit solution:
\begin{equation}\label{Eq-X}
X_t = \int_0^t\mathrm{e}^{\alpha(s-t)}dW_s.
\end{equation}
It is well know that there exists a Brownian motion $B$ such that
\begin{equation}\label{transform}
X_t = \frac{1}{\sqrt{2\alpha}}\mathrm{e}^{-\alpha t}B_{\mathrm{e}^{2\alpha t}-1} = \frac{1}{\sqrt{2\alpha}}\frac{B_{H(t)}}{\sqrt{H(t)+1}},
\end{equation}
where $H(t) = \mathrm{e}^{2\alpha t}-1$. For any stopping time $\tau$ of $B$, it is easy to check that $H^{-1}(\tau) = \log(1+\tau)/2\alpha$ is a stopping time of $X$. Thus it follows from Theorem \ref{main} that
\begin{equation*}
\begin{split}
\Enum\left[\sup_{0\leq t\leq\tau}f\left(\frac{1}{\sqrt{2\alpha}}\frac{|B_t|}{\sqrt{1+t}}\right)\right]
&\sim \Enum f(\log^{p/2}(1+\alpha H^{-1}(\tau))\\
&= \Enum f(\log^{p/2}(1+\frac{1}{2}\log(1+\tau)).
\end{split}
\end{equation*}
The desired result follows from the definition the moderate function and the fact that
\begin{equation*}
\log(1+\frac{1}{2}\log(1+x)) \sim \log(1+\log(1+x))
\end{equation*}
as $x\to 0$ or $x\to\infty$.
\end{proof}

Since any continuous local martingale is a time change of the Brownian motion, the above corollary implies a moderate maximal inequality for continuous local martingales.
\begin{corollary}
Let $M$ be a continuous local martingale vanishing at zero with quadratic variation process $[M]$. For any moderate function $f$, there exists two positive constants $c_f$ and $C_f$ such that for any stopping time $\tau$ of $M$,
\begin{equation*}
\begin{split}
c_f\Enum f\big(\log^{1/2}(1+\log(1+[M]_\tau))\big)
&\leq \Enum\left[\sup_{0\leq t\leq\tau}f\left(\frac{|M_t|}{\sqrt{1+[M]_t}}\right)\right]\\
&\leq C_f\Enum f\big(\log^{1/2}(1+\log(1+[M]_\tau))\big).
\end{split}
\end{equation*}
In particular, for any $p>0$, there exists two positive constants $c_p$ and $C_p$ independent of $M$ such that for any stopping time $\tau$ of $M$,
\begin{equation*}
c_p\Enum\log^{p/2}(1+\log(1+[M]_\tau)) \leq \Enum\left[\sup_{0\leq t\leq\tau}\frac{|M_t|^p}{(1+[M]_t)^{p/2}}\right] \leq C_p\Enum\log^{p/2}(1+\log(1+[M]_\tau)).
\end{equation*}
\end{corollary}

\begin{proof}
Since $M$ is a continuous local martingale, there exists a Brownian motion $B$ such that $M_t = B_{[M]_t}$. Thus the desired result follows directly from Corollary \ref{BM}.
\end{proof}

According to the BDG inequality, the $L^p$ maximum of a continuous local martingale $M$ in average behaves as $[M]^{p/2}$. The above corollary shows that the $L^p$ maximum of $M$, normalized by $(1+[M])^{p/2}$, in average behaves as $\log^{p/2}(1+\log(1+[M]))$. The relationship between the BDG inequality and our result is rather similar to that between
the central limit theorem and the law of iterated logarithm.

\section{Proof of Theorem \ref{main}}
Let $X^*$ be the maximum process of $|X|$ defined by
\begin{equation*}
X^*_t=\sup_{0\leq s\leq t}|X_s|.
\end{equation*}
To prove our main result, we first need two lemmas.
\begin{lemma}\label{upper}
There exists a function $\phi:\Rnum_+\to \Rnum_+$ satisfying $\phi(\delta)\to 0$ as $\delta\to 0$ such that for any $t\geq 1$ and $\delta>0$,
\begin{equation}
\Pnum(X^*_t<\delta\log^{1/2}t) \leq \phi(\delta).
\end{equation}
\end{lemma}

\begin{proof}
For any $x\geq 0$, let $\tau_x = \inf\{t\geq 0: |X_t|=x\}$. Following the proof for equation (2.30) in \cite{graversen1998maximal}, we define 
\begin{equation*}
u(x):= 2\int_0^x\mathrm{e}^{\alpha y^2}dy\int_0^y\mathrm{e}^{-\alpha z^2}dz.
\end{equation*}
It is easy to check that $u\in C^2(\Rnum)$ and satisfies the ordinary differential equation
\begin{equation*}
\frac{1}{2}u''(x)-\alpha xu'(x) = 1,\;\;\;u(0)=u'(0)=0.
\end{equation*}
By Ito's formula, we have
\begin{equation*}
\Enum u(X_{\tau_x\wedge t}) = \Enum\tau_x\wedge t+\Enum\int_0^{\tau_x\wedge t}u'(X_s)dW_s = \Enum\tau_x\wedge t.
\end{equation*}
Since $u$ is an even function, taking $t\to\infty$ in the above equation gives rise to
\begin{equation*}
\Enum\tau_x = u(x) = 2\int_0^x\mathrm{e}^{\alpha y^2}dy\int_0^y\mathrm{e}^{-\alpha z^2}dz
\leq \sqrt{\frac{\pi}{\alpha}}x\mathrm{e}^{\alpha x^2} \leq  C\mathrm{e}^{2\alpha x^2},
\end{equation*}
where $C$ is a constant independent of $x$. By Chebyshev's inequality, we have
\begin{equation*}
\Pnum(X^*_t<\delta\log^{1/2}t) = \Pnum(\tau_{\delta\log^{1/2}t}\geq t)
\leq t^{-1}\Enum\tau_{\delta\log^{1/2}t} \leq Ct^{2\alpha\delta^2-1}.
\end{equation*}
When $\delta$ is sufficiently small, for any $\epsilon>0$, we can find $T>1$ such that
\begin{equation*}
\sup_{t\geq T}\Pnum(X^*_t<\delta\log^{1/2}t) \leq CT^{-1/2} < \epsilon.
\end{equation*}
On the other hand, when $\delta$ is sufficient small,
\begin{equation*}
\sup_{1\leq t\leq T}\Pnum(X^*_t<\delta\log^{1/2}t) \leq \Pnum(X^*_1<\delta\log^{1/2}T) < \epsilon.
\end{equation*}
Combining the above two inequalities, we obtain the desired result.
\end{proof}

\begin{lemma}\label{lower}
There exists a function $\phi:\Rnum_+\to \Rnum_+$ satisfying $\phi(\delta)\to 0$ as $\delta\to 0$ such that for any $t\geq 2$ and $\delta>0$,
\begin{equation}
\Pnum(X^*_t\geq\delta^{-1}\log^{1/2}t) \leq \phi(\delta).
\end{equation}
\end{lemma}

\begin{proof}
By the law of iterated logarithm of the Brownian motion, we have
\begin{equation*}
\limsup_{t\to\infty}\frac{|B_t|}{\sqrt{2t\log\log t}} = 1,\;\;\;\textrm{a.s}.
\end{equation*}
Thus it follows from \eqref{transform} that
\begin{equation}\label{lim-X}
\limsup_{t\to\infty}\frac{|X_t|}{\log^{1/2}t}
= \frac{1}{\sqrt{2\alpha}}\limsup_{t\to\infty}\frac{|B_{\mathrm{e}^{2\alpha t}-1}|}{\mathrm{e}^{\alpha t}\log^{1/2}t}
= \frac{1}{\sqrt{\alpha}},\;\;\;\textrm{a.s}.
\end{equation}
Obviously, $\limsup_{t\to\infty}\frac{X^*_t}{\log^{1/2}t} \geq \frac{1}{\sqrt{\alpha}},$ a.s..  On the other hand, let $k(t):=\inf\{s: |X_s|=X^*_t\}$,  noticing any continuous function can attain its maximum over a closed interval, $k(t)$ is well-defined and $k(t)\leq t$. By \eqref{lim-X}, it is easy to see $k(t)\to\infty$ as $t\to\infty$, so  
$$\limsup_{t\to\infty} \frac{X^*_t}{\log^{1/2} t}\leq \limsup_{t\to\infty} \frac{X_{k(t)}}{\log^{1/2}k(t)}\leq \limsup_{t\to\infty} \frac{|X_t|}{\log^{1/2} t}=\frac{1}{\sqrt{\alpha}}. $$
To sum up, we have
\begin{equation*}
\limsup_{t\to\infty}\frac{X^*_t}{\log^{1/2}t} = \frac{1}{\sqrt{\alpha}},\;\;\;\textrm{a.s}.. 
\end{equation*}
When $\delta$ is sufficiently small, for any $\epsilon>0$, we can find $T\geq 2$ such that
\begin{equation*}
\sup_{t\geq T}\Pnum(X^*_t\geq\delta^{-1}\log^{1/2}t)
\leq \Pnum\left(\sup_{t\geq T}\frac{X^*_t}{\log^{1/2}t}\geq\frac{2}{\sqrt{\alpha}}\right) < \epsilon.
\end{equation*}
On the other hand, when $\delta$ is sufficient small, we have
\begin{equation*}
\sup_{2\leq t\leq T}\Pnum(X^*_t\geq \delta^{-1}\log^{1/2}t) \leq \Pnum(X_T^*\geq\delta^{-1}\log^{1/2}2) < \epsilon.
\end{equation*}
Combining the above two inequalities, we obtain the desired result.
\end{proof}

To proceed, we need a classical result, whose proof can be found in \cite[Chapter IV, Lemma 4.9]{revuz1999continuous}.
\begin{lemma}\label{goodlambda}
Let ${\phi\colon{\Rnum}_+\rightarrow{\Rnum}_+} $ satisfy ${\phi(\delta)\rightarrow 0}$ as ${\delta\rightarrow 0}$. Assume that a pair of nonnegative random variables $(X,Y)$ satisfies the following good $\lambda$ inequality for any $\delta,\lambda>0$:
\begin{equation*}
\Pnum(X\geq 2\lambda,Y<\delta\lambda) \leq \phi(\delta)\Pnum(X\geq\lambda).
\end{equation*}
Then for any moderate function $f$, there exists a positive constant $C$ depending on $f$ and $\phi$ such that
\begin{equation*}
\Enum f(X)\leq C\Enum f(Y).
\end{equation*}
\end{lemma}

In order to use the above lemma to prove our main results, we still need the following two lemmas.
\begin{lemma}\label{first}
There exists a function $\phi:\Rnum_+\to\Rnum_+$ satisfying $\phi(\delta)\to 0$ as $\delta\to 0$ such that for any stopping time $\tau$ with respect to $\{\mathscr{F}_t\}$ and any $\delta,\lambda>0$, the following good $\lambda$ inequality holds:
\begin{equation*}
\Pnum(\log^{1/2}(1+\alpha\tau)\geq 2 \lambda,X_\tau^*<\delta\lambda)
\leq \phi(\delta)\Pnum(\log^{1/2}(1+\alpha\tau)\geq\lambda).
\end{equation*}
\end{lemma}

\begin{proof}
It is easy to see that
\begin{equation*}
\Pnum_0(\log^{1/2}(1+\alpha\tau)\geq 2\lambda,X_\tau^*<\delta\lambda)
\leq \Pnum_0(\tau\geq r,X^*_s<\delta\lambda),
\end{equation*}
where $r = \alpha^{-1}(\mathrm{e}^{\lambda^2}-1)$ and $s = \alpha^{-1}(\mathrm{e}^{4\lambda^2}-1)$. By the Markov property of $X$, we have
\begin{equation*}
\begin{split}
\Pnum_0(\log^{1/2}(1+\alpha\tau)\geq 2 \lambda,X^*_\tau<\delta\lambda)
&\leq \Enum_0\left[1_{\{\tau\geq r\}}\Pnum_0(X^*_s<\delta\lambda|\mathscr{F}_r)\right]\\
&\leq \Enum_0\left[1_{\{\tau\geq r\}}\Pnum_{X_r}(X^*_{s-r}<\delta\lambda)\right]\\
&\leq \sup_{|x|<\delta\lambda}\Pnum_x(X^*_{s-r}<\delta\lambda)\Pnum_0(\tau\geq r)
\end{split}
\end{equation*}
For any $x\in\Rnum$, let $X^x$ be the solution to the following stochastic differential equation:
\begin{equation*}
dX^x_t = -\alpha X^x_tdt+dW_t,\;\;\;X^x_0 = x.
\end{equation*}
Then it is easy to check that $X^x_t = X^0_t+x\mathrm{e}^{-t}$ for any $x\in\Rnum$ and $t\geq 0$. This suggests that
\begin{equation*}
\sup_{|x|<\delta\lambda}\Pnum_x(X^*_{s-r}<\delta\lambda) \leq \Pnum_0(X^*_{s-r}<2\delta\lambda).
\end{equation*}
Thus we obtain that
\begin{equation*}
\begin{split}
\Pnum_0(\log^{1/2}(1+\alpha\tau)\geq 2\lambda,X^*_\tau<\delta\lambda)
\leq \Pnum_0(X^*_{s-r}<2\delta\lambda)\Pnum_0(\log^{1/2}(1+\alpha\tau)\geq\lambda).
\end{split}
\end{equation*}
For convenience, set $\lambda_0 = \log^{1/2}(\alpha+1)$. We first consider the case of $\lambda>\lambda_0$. In this case, we have $s-r\geq \mathrm{e}^{\lambda^2}\geq 1$. Thus it follows from Lemma \ref{upper} that
\begin{equation}\label{part11}
\Pnum_0(X^*_{s-r}<2\delta\lambda) \leq \Pnum_0(X^*_{\mathrm{e}^{\lambda^2}}<2\delta\lambda) \leq \phi(2\delta),
\end{equation}
where $\phi$ is the function defined in Lemma \ref{upper}. We next consider the case of $0<\lambda\leq\lambda_0$. In this case, we have $s-r = \alpha^{-1}(\mathrm{e}^{4\lambda^2}-\mathrm{e}^{\lambda^2}) \geq c\lambda^2$, where $c$ is a positive constant independent of $\lambda$. It is easy to see from \eqref{system} that for any $t\geq 0$,
\begin{equation*}
(1+\alpha t)X^*_t \geq W^*_t.
\end{equation*}
This suggests that
\begin{equation}\label{part12}
\begin{split}
\Pnum_0(X^*_{s-r}<2\delta\lambda) &\leq \Pnum_0(X_{c\lambda^2}^*<2\delta\lambda)
\leq \Pnum_0(W^*_{c\lambda^2}<2(1+\alpha c\lambda_0^2)\delta\lambda)\\
&\leq \Pnum_0(|W_{c\lambda^2}|<2(1+\alpha c\lambda_0^2)\delta\lambda)
= \Pnum_0(|W_c|<2(1+\alpha c\lambda_0^2)\delta),
\end{split}
\end{equation}
which tends to zero as $\delta\to 0$. Combining \eqref{part11} and \eqref{part12}, we obtain the desired result.
\end{proof}

\begin{lemma}\label{second}
There exists a function $\phi:\Rnum_+\to\Rnum_+$ satisfying $\phi(\delta)\to 0$ as $\delta\to 0$ such that for any stopping time $\tau$ with respect to $\{\mathscr{F}_t\}$ and any $\delta,\lambda>0$, the following good $\lambda$ inequality holds:
\begin{equation*}
\Pnum(X^*_\tau\geq 2\lambda,\log^{1/2}(1+\alpha\tau)<\delta\lambda) \leq \phi(\delta)\Pnum(X^*_\tau\geq\lambda).
\end{equation*}
\end{lemma}

\begin{proof}
Let $\tau_\lambda = \inf\{t\geq 0: |X_t|=\lambda\}$. It is easy to see that
\begin{equation*}
\Pnum_0(X^*_\tau\geq 2\lambda,\log^{1/2}(1+\alpha\tau)<\delta\lambda)
\leq \Pnum_0(X^*_{s\vee\tau_\lambda}\geq 2\lambda,\tau>\tau_\lambda),
\end{equation*}
where $s = \alpha^{-1}(\mathrm{e}^{\delta^2\lambda^2}-1)$. By the strong Markov property of $X$, we have
\begin{equation*}
\begin{split}
\Pnum_0(X^*_\tau\geq 2\lambda,\log^{1/2}(1+\alpha\tau)<\delta\lambda)
&\leq \Enum_0\left[1_{\{\tau>\tau_\lambda\}}
\Pnum_0(X^*_{s\vee\tau_\lambda}\geq 2\lambda|\mathscr{F}_{\tau_\lambda})\right]\\
&\leq \Enum_0\left[1_{\{\tau>\tau_\lambda\}}
\Pnum_{X_{\tau_\lambda}}(X^*_{s\vee\tau_\lambda-\tau_\lambda}\geq 2\lambda)\right]\\
&\leq \sup_{|x|=\lambda}\Pnum_x(X^*_{s\vee\tau_\lambda-\tau_\lambda}\geq 2\lambda)\Pnum_0(\tau>\tau_\lambda).
\end{split}
\end{equation*}
Since $X^x_t = X^0_t+x\mathrm{e}^{-t}$ for any $x\in\Rnum$ and $t\geq 0$, it is easy to check that
\begin{equation*}
\sup_{|x|=\lambda}\Pnum_x(X^*_{s\vee\tau_\lambda-\tau_\lambda}\geq 2\lambda)
\leq \Pnum_0(X^*_{s\vee\tau_\lambda-\tau_\lambda}\geq \lambda) \leq \Pnum_0(X^*_s\geq \lambda).
\end{equation*}
This shows that
\begin{equation*}
\begin{split}
\Pnum_0(X^*_\tau\geq 2\lambda,\log^{1/2}(1+\tau)<\delta\lambda)
&\leq \Pnum_0(X^*_s\geq \lambda)\Pnum_0(X_\tau^*\geq\lambda).
\end{split}
\end{equation*}
For convenience, set $\lambda_0 = \log^{1/2}(2\alpha+\alpha^{-1}+1)$. We first consider the case of $\delta\lambda>\lambda_0$. In this case, we have $2\leq s\leq \mathrm{e}^{4\delta^2\lambda^2}$. Thus it follows from Lemma \ref{lower} that
\begin{equation}\label{part21}
\Pnum_0(X^*_s\geq\lambda) \leq \Pnum_0(X^*_{\mathrm{e}^{4\delta^2\lambda^2}}\geq\lambda)\leq \phi(2\delta),
\end{equation}
where $\phi$ is the function defined in Lemma \ref{lower}. We next consider the case of $0<\delta\lambda\leq\lambda_0$. In this case, we have $s\leq C\delta^2\lambda^2\leq C\lambda_0^2$, where $C$ is a positive constant independent of $\delta$ and $\lambda$. By Gronwall's inequality, it is easy to see from \eqref{system} that for any $t\geq 0$,
\begin{equation*}
X^*_t \leq \mathrm{e}^{\alpha t}W^*_t.
\end{equation*}
This suggests that
\begin{equation}\label{part22}
\begin{split}
\Pnum_0(X^*_s\geq\lambda) &\leq \Pnum_0(X^*_{C\delta^2\lambda^2}\geq\lambda)
\leq \Pnum_0(W^*_{C\delta^2\lambda^2}\geq\lambda \mathrm{e}^{-\alpha C\lambda_{0} ^2})\\
&\leq \Pnum_0(|W_{C\delta^2\lambda^2}|\geq\lambda \mathrm{e}^{-\alpha C\lambda_{0} ^2})
= \Pnum_0(|W_C|\geq\delta^{-1}\mathrm{e}^{-\alpha C\lambda_{0} ^2}),
\end{split}
\end{equation}
which tends to zero as $\delta\to 0$. Combining \eqref{part21} and \eqref{part22}, we obtain the desired result.
\end{proof}

We are now in a position to prove our main theorem.
\begin{proof}[Proof of Theorem \ref{main}]
The first part of the theorem follows directly from Lemmas \ref{goodlambda}, \ref{first}, and \ref{second}. We next prove the second part of the theorem. Let $Y$ be the Ornstein-Uhlenbeck process solving the stochastic differential equation
$$
dY_t = -Y_tdt+d\tilde W_t,\;\;\;Y_0 = 0, 
$$
where $\tilde W_t=\sqrt{\alpha} W_{t/\alpha}$. By the explicit expression \eqref{Eq-X} and basic calculation, one can see $X_t = Y_{\alpha t}/\sqrt{\alpha}$.   
And for any stopping time $\tau$ of $X$, it is easy to check that $\alpha\tau$ is a stopping time of $Y$. If we take $f(t) = t^p$ with $p>0$, then there exists positive constants $c_p$ and $C_p$ independent of $\alpha$ such that
\begin{equation*}
c_p\Enum\log^{p/2}(1+\alpha\tau) \leq \Enum\left[\sup_{0\leq t\leq\tau}|Y_{\alpha t}|^p\right] \leq C_p\Enum\log^{p/2}(1+\alpha\tau),
\end{equation*}
which gives the desired result.
\end{proof}

\section*{Acknowledgements}
The authors gratefully acknowledge X. Chen and X. Xu for stimulating discussions. This work was supported by National Postdoctoral Program for Innovative Talents (201600182) of China.

\setlength{\bibsep}{5pt}
\small\bibliographystyle{urnst}

\end{document}